\documentclass[13pt]{article}
\usepackage[centertags]{amsmath}
\usepackage{amsfonts}
\usepackage{amssymb}
\usepackage[symbol]{footmisc}
\usepackage{amsthm}
\usepackage{tikz-cd}
\input{amssym.def}

\newtheorem{theorem}{Theorem}[section]
\newtheorem{Theorem}[theorem]{Theorem}
\newtheorem{Lemma}[theorem]{Lemma}
\newtheorem{Proposition}[theorem]{Proposition}

\theoremstyle{definition}
\newtheorem{Definition}[theorem]{Definition}
\newtheorem{Example}[theorem]{Example}

\theoremstyle{remark}

\numberwithin{equation}{section}

\newcommand{\lc}{\mathcal{L}}
\newcommand{\rc}{\mathcal{R}}
\newcommand{\hc}{\mathcal{H}}
\newcommand{\jc}{\mathcal{J}}
\newcommand{\ec}{\mathcal{N}}

\title{\Large \bf On completely regular and Clifford  ordered semigroups}
\author{A. K. Bhuniya and K. Hansda \footnote{correspondingauthor} \\
\normalsize{Department of Mathematics, Visva-Bharati
University,}\\
\normalsize{Santiniketan, Bolpur - 731235, West Bengal, India}\\
\normalsize{anjankbhuniya@gmail.com}, \
\normalsize{kalyanh4@gmail.com $^*$}}
\date{}

\begin{document}

\maketitle

\begin{abstract}{\footnotesize}
Lee and Kwon \cite{LK}  defined an ordered semigroup $S$ to be
completely regular if $a \in (a^2Sa^2]$ for every $a \in S$. We
characterize every completely regular ordered semigroup as a union
of $t$-simple subsemigroups, and every Clifford ordered semigroup as
a complete semilattice of $t$-simple subsemigroups. Green's Theorem
for the completely regular ordered semigroups has been established.
In an ordered semigroup $S$, we call an element $e$ an ordered
idempotent if it satisfies $e \leq e^2$. Different characterizations
of the regular, completely regular and Clifford ordered semigroups
are done by their ordered idempotents. Thus a foundation for the
completely regular ordered semigroups and Clifford ordered
semigroups has been developed.
\end{abstract}

\section{Introduction}
Ordered semigroups bring the opportunity to study a partial order
together with an associative binary operation, two most simple
algebraic structures on the same set. Simplicity in their definition
makes the ordered semigroups frequent to appear in several branches
of not only in mathematics but also in the whole area of our study
ranging from computer science to social science to economics, on the
other hand, it makes them difficult to characterize.  Contrary to
what one might expect, the passage from the semigroup to the ordered
semigroup case is not straightforward. As an instance, it may be
mentioned that till now we don't have any formulation for the
ordered factor semigroup. Care must be taken to choose the proper
definitions and to justify that the definitions chosen are proper.

There are several articles on ordered semigroups, topological
ordered semigroups etc. Probably the huge impact of regular rings
and semigroups have been influenced the researchers  to introduce
the  natural partial order  on regular semigroups as well as to
introduce a  natural notion of regularity on a partially ordered
semigroup,  which arises out of a beautiful combination of the
partial order and binary operation. Let us call this second kind of
regularity as ordered regularity and semigroups in which every
element is ordered regular as regular ordered semigroups (reason
behind such naming is that regularity is introduced on ordered
semigroups and hence the term regular qualifies the ordered
semigroups), and the first kind as ordered regular semigroups
(because a partial order is considered on a regular semigroup).

T. Saito studied systematically the influence of order on regular,
idempotent, inverse, and completely regular semigroups
\cite{Saito1962} - \cite{Saito1971}, whereas Kehayopulu, Tsingelis,
Cao and many others characterized regularity, complete regularity,
etc. on ordered semigroups \cite{Cao1999} - \cite{CX2000},
\cite{Ke1990}-\cite{LK}. Success attained by the school
characterizing regularity on ordered semigroups are either in the
semilattice and complete semilattice decompositions into different
types of simple components, viz. left, t-, $\sigma$,
$\lambda$-simple etc. or in its ideal theory.

Complete regularity  on ordered semigeoups was introduced by Lee and
Kwon \cite{LK}. Here we give their  complete semilattice
decomposition and express them as a union of $t$-simple ordered
semigroups. This supports the observation of Cao \cite{CX2000} that
the $t$-simple ordered semigroups plays the same role in the theory
of ordered semigroups as groups in the theory of semigroups without
order. Then it follows that the semigroups which are semilattices of
$t$-simple ordered semigroups are the analogue of Clifford
semigroups. Though it is not under the name Clifford ordered
semigroups, but such ordered semigroups have been studied
extensively by Cao \cite{Cao2000} and Kehayopulu \cite{Ke MJ 92},
specially complete semilattice decomposition of such semigroups. We
generalize such ordered semigroups into left Clifford ordered
semigroups. Another successful part of this paper is identification
of the ordered idempotent elements in an ordered semigroup and
exploration of their behavior in both completely regular and
Clifford ordered semigroups. Also an extensive study has been done
on the idempotent ordered semigroups. Complete semilattice
decomposition of these semigroups automatically suggests the looks
of rectangular idempotent semigroups and in this way we arrive to
many other important classes of idempotent ordered semigroups.

The presentation of the article is as follows. This section is
followed by preliminaries. In Section 3, basic properties of  the
$t-$simple ordered semigroups which we call here group like ordered
semigroups have been studied.  Completely regular ordered semigroups
have been characterized in Section 4. Section 5 is devoted to the
the Clifford ordered semigroups and their generalizations.

\section{Preliminaries}
An ordered semigroup  is a partially ordered set $(S, \leq)$, and at
the same time a semigroup $(S,\cdot)$ such that $ \textrm{for all}
\; a, b , c \in S;  \;a \leq b \Rightarrow
 \;ca \leq cb \;\textrm{and} \;ac \leq bc$.
It is denoted by $(S,\cdot, \leq)$. Throughout this article, unless
stated otherwise, $S$ stands for an ordered semigroup and we assume
that $S$ does not contain the zero element.

An equivalence relation $\rho$ on $S$ is called left (right)
congruence if for every $a, b, c \in S; \;a \;\rho \;b \;
\textrm{implies} \;ca \;\rho \;cb \;(ac \;\rho \;bc)$. By a
congruence we mean both left and right congruence. A congruence
$\rho$ is called a semilattice congruence on $S$ if for all $a, b
\in S, \;a \;\rho \;a^{2} \;\textrm{and} \;ab \;\rho \;ba$. By a
complete semilattice congruence we mean a semilattice congruence
$\sigma$ on $S$ such that for $a, b \in S, \;a \leq b$ implies that
$a \sigma ab$.

For every subset $H\subseteq S$, denote $(H] := \{t \in S : t \leq
h, \;\textrm{for \;some} \;h \in H \}$.

An element $a \in S$ is called  ordered regular \cite{ke91}(left
regular \cite{Ke1990}) if  $ \;a \in (aSa] \;(a \in (Sa^2])$.  An
element $b \in S$ is inverse of $a$ if $a \leq aba \;\textrm{and}
\;b \leq bab$. We denote the set of all inverse elements of $a$ in
$S$  by $ \;V_\leq(a)$.

Let $I$ be a nonempty subset of $S$. Then $I$ is called a left
(right) ideal of $S$, if $SI \subseteq I$ ($IS \subseteq I$) and
$(I] \subseteq I$. If $I$ is both a left and a right ideal, then it
is called an ideal of $S$. We call $S$ a (left, right) simple
ordered semigroup if it does not contain any proper (left, right)
ideal. If $S$ is both left simple and right simple, then it is
called $t$-simple.

For $a \in S$, the smallest (left, right) ideal of $S$ that contains
$a$ is denoted by ($L(a), R(a)$) $I(a)$. It is easy to verify that
on a regular ordered semigroup $S$, $ L(a) = (Sa] = \{x \in S \mid x
\leq sa, \; s \in S\}$. Similarly for $R(a)$ and $I(a)$.

Kehayopulu \cite{ke91} defined Green's relations on a regular
ordered semigroup $S$ as follows:
 $$ a \lc b   \; if   \;L(a)= L(b),  \;a \rc b   \; if   \;R(a)= R(b), \;a \jc b   \; if   \;I(a)= I(b), \;\textrm{and} \;\hc= \;\lc \cap
\;\rc.$$
 These four relations $\lc, \;\rc, \;\jc \;\textrm{and} \;\hc$ are
equivalence relations.

A subset $F$ of $S$ is called a filter if for $a,b \in S, \;c \in
F$; $(i) \;ab \in F$ implies that $a \in F \;\textrm{and} \;b \in
F$, and $(ii) \;c \leq a $ implies that $c \in F$.  The smallest
filter containing $a \in S$ is  denoted by $N(a)$.

In \cite{ke91}, Kehayopulu defined a binary relation $\ec$ on $S$
by: for $a, b \in S$,
 $a \ec b \; if \;N(a)=N(b)$.
She proved that $\ec$ is a semilattice congruence and gave an
example  \cite{ke91} to show that this is not the least semilattice
congruence on $S$. In fact, $\ec$ is the least complete semilattice
congruence on $S$ \cite{Gao1998}.

An ordered semigroup $S$  is called complete semilattice of
subsemigroup of type $\tau$ if there exists a complete semilattice
congruence $\rho $ such that $(x)_{\rho}$ is a type $\tau$
subsemigroup of $S$. Equivalently \cite{Ke1990}, there exists a
semilattice $Y$ and a family of subsemigroups $\{S_\alpha\}_{\alpha
\in Y}$ of type $\tau$ of $S$ such that:
\begin{enumerate}
\item \vspace{-.4cm}
$S_{\alpha}\cap S_{\beta}= \;\phi$ for every $\alpha, \;\beta \in
\;Y \;\textrm{with} \; \alpha \neq \beta,$
\item \vspace{-.4cm}
$S=\bigcup _{\alpha \;\in \;Y} \;S_{\alpha},$
\item \vspace{-.4cm}
$S_{\alpha}S_{\beta} \;\subseteq \;S_{\alpha \;\beta}$ for any
$\alpha, \;\beta \in \;Y,$
\item \vspace{-.4cm}
$S_{\beta}\cap (S_{\alpha}]\neq \phi$ implies $\beta \;\preceq
\;\alpha,$ where $\preceq$ is the order of the semilattice $Y$
defined by
$$ \preceq:=\{(\alpha,\;\beta)\;\mid
\;\alpha=\alpha\;\beta\;(\beta\;\alpha)\}. $$
\end{enumerate}

An ordered semigroup $(S, \cdot, \leq)$ is called a semilattice
ordered semigroup if $a \vee b$ exists in the poset $(S, \leq )$ for
every $a, b \in S$. In this case, $ a(b \vee c)=ab \vee ac \;
\textrm{and} \; (a \vee b)c=ac \vee bc $ for every $a, b, c \in S$.

If $F$ is a semigroup, then the set $P_f(F)$ of all finite subsets
of $F$ is a semilattice ordered semigroup with respect to the
product $'\cdot'$ and partial order relation $'\leq'$ given by: for
$A, B \in P_f(F)$,
$$ A \cdot B = \{ab \mid a \in A, b \in B\} \;
\textrm{and} \; A \leq B \; \textrm{if and  only  if} \; A \subseteq
B. $$

Now we show that this semilattice ordered semigroup $P_f(F)$ has the
universal mapping property in the following sense:
\begin{Proposition}
Let $F$ be a semigroup, $S$ be a semilattice ordered semigroup and
$f : F \longrightarrow S$ be a semigroup homomorphism. Then there is
a ordered semigroup homomorphism $\phi : P_f(F) \longrightarrow S$
such that the following diagram is commutative:
\begin{center}
\begin{tikzcd}[row sep=3.5em, column sep=3.5em]
F \arrow{r}{l} \arrow[swap]{dr}{f} & P_f(F) \arrow{d}{\phi} \\
& S
\end{tikzcd}
\end{center}
where $l : F\longrightarrow P_f (F)$ is given by $l(x) = \{x\}$.
\end{Proposition}
\begin{proof}
Define $\phi : P_f(F) \longrightarrow S$ by: for $A \in P_f(F)$, $
\phi(A)=\vee_{a \in A} f(a)$. Then for every $A, B \in P_f(F),
\;\phi(AB)= \vee_{a \in A, b \in B}f(ab)= \vee_{a \in A, b \in
B}f(a)f(b)= (\vee_{a \in A}f(a))(\vee_{b \in B}f(b))= \phi(A)
\phi(B),$ and if $A \leq B$, then $\phi(A) = \vee_{a \in A}f(a) \leq
\vee_{b \in B}f(b) = \phi(B)$ shows that $\phi$ is an ordered
semigroup homomorphism. Also $\phi \circ l = f$.
\end{proof}

For the notions of semigroups (without order), we refer to Howie
\cite{Howie1995},  and Petrich  and Reilly \cite{PRbook}.

\section{Group like ordered semigroups}
A group $G$ can be considered as a semigroup such that for every $a,
b \in G$,  the equations $a = xb \;\textrm{and} \;a = by$ have
solutions in $G$. Thus a semigroup $S$ is a group if and only if it
is t-simple.

Also we have  following two significant observations. First of which
explores a natural analogy between groups and t-simple ordered
semigroups.
\begin{Proposition}\label{cr1}
A semigroup $F$ is a group if and only if the ordered semigroup
$P_f(F)$  is a t-simple ordered semigroup.
\end{Proposition}

\begin{proof}
First suppose that  $F$ is a group, and  $A, B \in P_f(F)$. Then for
each $a \in A$ and $b \in B$ there are unique $x, y \in F$ such that
$a= xb$ and $a= by$. Let us denote them by $x_{a,b}$ and $y_{a,b}$
respectively. Then $X= \{ x_{a,b} | a \in A, \;b \in B\} $ and $Y=
\{y_{a,b} | a \in A, \;b \in B\} $ are in $P_f(F)$ such that $A
\subseteq XB \;\textrm{and}  \;A \subseteq BY $, that is $A \leq XB$
and $A \leq BY $. Thus the ordered semigroup $P_f(F)$ is t-simple.

Conversely, assume  that $a, b \in F$. Then both $A= \{a\}$ and $B=
\{b\}$ are elements of $P_f(F)$. Then  there exist $X, Y \in P_f(F)$
such that $A \leq XB$ and $A \leq BY$, that is $A\subseteq XB$ and
$A \subseteq BY $. Hence there are $x \in X, \;y \in Y$ such that
$a= xb$ and $ a= by $, which shows that $F$ is a group.
\end{proof}

Our second observation is that every t-simple ordered semigroup is
regular. Consider a t-simple ordered  semigroup $S$ and let $a$ be
an element of  $S$. Then there are $t, x \in S$ such that $a\leq at$
and $t \leq xa $, which implies that $a \leq axa$.

From the  above observations and according to the context of this
article we wish to call here the t-simple ordered semigroups as
group like ordered semigroups.
\begin{Definition}
An ordered semigroup $S$ is called a group like ordered semigroup
 if for all $a, b \in S \;\textrm{there \;are} \;x, y \in S
\;\textrm{such \;that} \;a \leq xb \;and \;a \leq by$.
\end{Definition}
We further generalize  this structure to left and right group like
ordered semigroups.
\begin{Definition}
A  regular ordered semigroup $S$ is called a left group like ordered
semigroup if for all $a, b \in S \;\textrm{there \;is} \;x \in S$
such that $a \leq xb $.

Right group like ordered semigroup are defined  dually.
\end{Definition}
Thus  an ordered semigroup $S$ is   group like  ordered semigroup if
and only if it is both a left  group like ordered semigroup and a
right group like ordered semigroup.

Following correspondence between   group and left group like ordered
semigroups can be proved similarly to Proposition \ref{cr1}.
\begin{Proposition}\label{cr2}
A semigroup $F$ is a left group if and only if the ordered semigroup
 $P_f(F)$ of all finite subsets of $F$ is a left group like ordered semigroup.
\end{Proposition}
This result follows from the observation that a semigroup $F$ is a
left group if and only if $F$ is both regular and left simple.
\begin{Theorem}\label{cr3}
Let S be an ordered semigroup. Then
\begin{enumerate}
  \item \vspace{-.4cm}
 $S$ is a  group like ordered semigroup if and only if $a \in (bSb]$ for all $a, \;b \in S$.
 \item \vspace{-.4cm}
$S$ is  left group like ordered semigroup if and only if $ a
\in(aSb]$ for all $a, \;b \in S$.
\end{enumerate}
\end{Theorem}
\begin{proof}
$(1)$ Let $a, b \in S$. Since  $S$ is a group like ordered semigroup
there are $t, x \in S$ such that $a \leq tb$ and $t \leq bx$, which
implies  that $a \leq bxb$ and so $a \in (bsb]$.

The converse is trivial.

The assertion $(2)$ can be proved similarly.
\end{proof}

Let $a \in S$ be an ordered regular element, then there is  $x \in
S$ be such that $a \leq axa$. This yields that  $ax \leq (ax)^2$ and
$xa \leq (xa)^2$.

Thus in a  regular ordered semigroup there are elements $e \in S$
such that $e \leq e^2$. Later we show that these elements are as
efficient  to describe the structure of  regular ordered semigroups
as idempotents in a regular semigroup without order.
\begin{Definition}
Let $S$ be an ordered semigroup. An element $e \in S$ is called an
ordered idempotent if $e \leq e^2$.
\end{Definition}
We denote the set all ordered idempotents of an ordered semigroup
$S$ by $E_\leq(S)$.

\begin{Lemma}\label{BI13}
Let $L$ be a left and $R$ be a right ideal of regular ordered
semigroup $S$. Then  for every $e, f \in E_\leq (S)$:

\begin{enumerate}
 \item \vspace{-.4cm}
$L \cap (eS] = (eL]$.
\item \vspace{-.4cm}
$R \cap (Se] = (Re]$.
\item \vspace{-.4cm}
$(Sf] \cap (eS] = (eSf]$.
\end{enumerate}
\end{Lemma}
\begin{proof}
$(1)$ We have $(eL] \subseteq L$, since $L$ is a left ideal of $S$.
Also $(eL] \subseteq (eS]$. Thus $(eL] \subseteq L \cap (eS]$. Let
$y \in L\cap(eS]$. Then there is $s \in S$ such that $y \leq es$.
Since $S$ is regular there is $z \in S$ such that $$y \leq yzy \leq
ezy.$$ Since $L$ is left ideal of $S$, $zy \in L$ and hence
$L\cap(eS] \subseteq (eL]$. Thus $(eL] = L \cap (eS]$.

$(2)$ This is similar to $(1)$.

$(3)$ Let $e, f \in E_\leq(S)$. Then $(eSf] \subseteq (Sf] \cap
(eS]$. Now  consider $z \in(Sf] \cap (eS]$. Then there are $s, t \in
S$ such that $x \leq sf$ and $x \leq et$. By the regularity of $S$,
there is $w \in S$ such that
\begin{align*}
x &\leq xwx\\
  & \leq etwsf.
\end{align*}
Thus $(Sf] \cap(eS] = (eSf]$.
\end{proof}

A group (without order) contains exactly one idempotent. In group
like ordered semigroups such uniqueness does not occur.

\begin{Example}
Consider the ordered semigroup $(\mathbb{R^+}, \;\cdot, \;\leq)$.
Then every positive integer $a \geq 1$ is an ordered idempotent.
\end{Example}

Though there may have many ordered idempotents in a group like
ordered semigroup, now we show that they are related in a meaningful
way.

\begin{Theorem}\label{cr5}
A  regular ordered semigroup $S$ is a  group like ordered semigroup
if and only if for all $e,f \in E_{\leq} (S), \;e \hc f$.
\end{Theorem}
\begin{proof}
First suppose that $S$ is  a  group like ordered semigroup. Then
clearly any two elements of $S$ are $\hc-$related . Thus in
particular $e \hc f$ for all $e, \;f \in E_{\leq}(S)$.

Conversely, assume  that $S$ is a regular ordered  semigroup that
satisfies the given conditions. Let $a, b \in S$. Since $S$ is
regular, there are $s, t \in S$ such that $a \leq asa$ and $b \leq
btb$. Then  $as, sa, bt, tb \in E_{\leq}(S)$ and hence $as \leq btu$
and $sa \leq vtb$ for some $u, v \in S$. Now  $a \leq asa$ implies
that $a \leq btua \leq bt_1, \; \textrm{where} \;t_1= \;tua,
\;\textrm{and} \;a \leq avtb \leq t_2 b  \; \textrm{where} \;t_2=
\;avt$. Thus $S$ is a group like ordered semigroup.
\end{proof}

Thus what we mean in a regular semigroup by having unique
idempotent, the same is meant in a regular ordered  semigroup by the
ordered idempotents are in  same $\hc-$class.

In the following we have another similar result which is analogous
to the result that a regular semigroup is inverse if and only if the
idempotents commutate.
\begin{Theorem}\label{cr6}
Let $S$ be a regular ordered  semigroup. Then for every $a \in S$
and $a', \;a'' \in V_{\leq}(a)$, $a' \hc a''$ if and only if for
every $e, \;f \in E_{\leq}(S)$ there is $x \in S$ such that $ef \leq
fxe$.
\end{Theorem}
\begin{proof}
First suppose that for every $a \in S$ and  $a', \;a'' \in
V_{\leq}(a), \;a' \hc a''$. Consider $e, f \in E_{\leq}(S)$. Since
$S$ is regular we have  $ V_{\leq}(ef)\neq \phi$. Let  $x \in
V_{\leq}(ef)$. Then $x  \leq  x ef x$ implies that $$fxe \leq fxe^2
f^2xe \;\textrm{and} \;ef \leq ef x ef. $$

Then  $ef \leq ef^2 x e^2f$. Thus $ef \in V_{\leq}(fxe)$. Again  $x
\leq x ef x$ yields that $fxe \leq (fxe)^2$, that is,   $fxe \in
E_{\leq}(S)$ and so $fxe \in V_{\leq}(fxe)$. Hence $ef \hc fxe$.
Then there are $u, v \in S$ such that $$ef \leq fxeu \;\textrm{and}
\;ef \leq vfxe.$$ Now $ef \leq ef x ef$ implies that $ef \leq f(xe
uxv fx)e$.

Conversely, assume  that for all $\;e,\;f \in E_{\leq}(S)$ there is
$x \in S$ such that $$ef \in (fSe]  \;\textrm{and} \;a', a'' \in
V_{\leq}(a).$$ Then $aa'', \;aa' \in E_{\leq}(S)$ which implies that
$$aa''aa' \leq aa' x aa'' \;\textrm{for some} \;x \in S.$$  Then we have
\begin{align*}
a' &\leq a'aa'\\
   & \leq a'(aa''a) a'\\
   & \leq a'a  a'x aa''\leq t a'',
\;\textrm{where} \;t = a'a a'x a \in S.
\end{align*}
 Similarly there is $s \in S$
such that $a'' \leq sa'$. Thus $a' \lc a''$ and similarly $a'' \rc
a'$. Hence $a' \hc a''$.
\end{proof}

We introduce inverse ordered semigroups as follows:
\begin{Definition}
A regular ordered semigroup $(S, \cdot, \leq)$ is called inverse if for every $a \in S$ and $a', a'' \in V_{\leq}(a)$, $a' \hc a''$.
\end{Definition}

Inverse ordered semigroups have been studied in \cite{HJ1}

In a semigroup $S$ two elements $a, b$ are said to be
$\hc-$commutative if $ab =bxa$ for some $x \in S$. We define
$\hc-$commutativity in an ordered semigroup as follows.
\begin{Definition}
Let $S$ be an ordered semigroup and let $a, b \in S$. Then $a, b$
are said to $\hc-$commutative if  $ab \leq bxa$ for some  $x \in S$.
\end{Definition}
An ordered semigroup $S$ is called $\hc-$commutative if every  $a, b
\in S$ are $\hc-$commutative.

\section{Completely regular ordered semigroups}
Every completely regular semigroup is a semilattice of completely
simple semigroups and a regular semigroup S is completely regular if
and only if it is a union of groups. The ordered semigroup $P_f(F)$
is a  regular ordered semigroup when $F$ is regular,  where as if
$F$ is a group then $P_f(F)$ is group like ordered semigroup. This
motivates us to characterize all regular ordered semigroups which
are union of group like ordered semigroups.  Lee and Kwon \cite{LK}
defined completely regular ordered semigroups  as follows:
\begin{Definition}
An element $a$  of an ordered semigroup $S$ is called completely
regular if  $a \in (a^{2} S a^{2}]$.
\end{Definition}
We denote the set of completely regular elements of an ordered
semigroup $S$ by $Gr_\leq(S)$. An ordered semigroup $S$ is called
completely regular if if for every $a \in S, \;a \in  (a^2 S a^2]$.

Immediately we have the following result.
\begin{Theorem}\label{cr8}
An  $\hc$-commutative ordered semigroup $S$ is regular if and only
if it is completely regular.
\end{Theorem}
\begin{proof}
First suppose that $S$ is regular. Let $a \in S$. Then there is $x
\in S$ such that $a \leq axa x axa $. Since $S$ is $\hc-$commutative
there are $x', y' \in S$ such that $xa \leq ax'x$ and $ax \leq
xy'a$. Then $a \leq axa x axa $ gives $a \leq a^2 x'x^3 y'a^2$.
Therefore $S$ is completely regular.

Converse follows trivially.
\end{proof}
Thus $S$ is a completely regular ordered semigroup if and only if
for every $a \in S$ there is $s \in S$ such that $a \leq a^{2} s
a^{2}$. The following proposition not only justifies such
observation but also shows that the size of the class of all
completely regular ordered semigroups is not less than that of the
class of all completely regular semigroups.

\begin{Proposition}\label{cr9}
Let $F$ be  a  semigroup. Then the ordered semigroup $P_f(F)$ of all
subsets of $F$  is a completely regular ordered semigroup if and
only if $F$ is a completely regular semigroup.
\end{Proposition}
\begin{proof}
First suppose that $F $ is a completely regular  semigroup. Consider
a finite subset $A$ of  $F$. Then for each $a \in A$, there is
$s_{a} \in F$ such that $a \leq a^{2} s_{a} a^{2}$. Now $X =
\{s_{a}| a \in A\} \in P_f(F)$ is such that $A \subseteq A^{2} X
A^{2} \;i.e \;A \leq A^{2} X A^{2}$. Thus  $ P_f(F)$ is completely
regular.

Conversely, assume that $S= P_f(F)$ is a completely regular ordered
semigroup. Let $a \in F$. Then for $A = \{a \}\in S$ there exists $X
\in S$ such that $A \leq A^{2} X A^{2}$. This implies that there is
$x \in X$ such that $a = a^{2} x a^{2}$ and hence $F$ is a
completely regular semigroup.
\end{proof}
Let $S$ be a group like ordered semigroup. Then for each $a \in S$
there exist $s, t \in S$ such that $a \leq s a^{2} \;\textrm{and}
\;s \leq a^{2}t$. This implies that $a \leq a^{2} t a^{2}$. Thus
every group like ordered semigroup is a completely regular ordered
semigroup.
\begin{Lemma}\label{cr10.1}
Let $S$ be completely regular ordered semigroup. Then for every $a
\in S$ there is $x \in S$ such that $a \leq axa^2$ and $a \leq a^2
xa$.
\end{Lemma}
\begin{proof}
Let $a \in S$. Then there is $t \in S$ such that $$a \leq a^2ta^2
\leq a^2(a^2ta^2ta^2ta^2ta^2)=a^2xa; \;\textrm{where} \;x=
a^2ta^2ta^2ta^2ta^2.$$ Similarly $a \leq axa^2$. This completes the
proof.
\end{proof}

Following  equivalent conditions to complete regularity can be
proved easily.

\begin{Theorem}\label{cr10}
In an ordered semigroup $S$ the following conditions are equivalent:
\begin{enumerate}
  \item \vspace{-.4cm}
   $S$ is completely regular;
   \item \vspace{-.4cm}
    $a \in (a^{2} S a] \cap (a Sa^{2}]$ for all $a \in S$;
    \item \vspace{-.4cm}
    $a \in (a^{2} S a] \cap (Sa^{2}]$ for all $a \in S$;
   \item \vspace{-.4cm}
   $a \in (aSa^{2}] \cap (a^{2}S]$ for all $a \in S$;
  \item \vspace{-.4cm}
   $S$ is regular ordered  semigroup and $a \in (a^{2}S] \cap (Sa^{2}]$ for all $a \in
   S$.
\end{enumerate}
\end{Theorem}
Now we show that every element $a$ of a completely regular ordered
semigroup has an ordered inverse element $a'$ that $\hc-$commutes
with $a$.
\begin{Theorem}\label{cr11}
An ordered semigroup $S$ is  completely regular if and only if for
all $a \in S$ there exists $a' \in V_{\leq}(a)$ such that $aa' \leq
a'ua \;and \;a'a \leq ava' \;for \;some  \;u, v \in S$.
\end{Theorem}
\begin{proof}
First assume that $S$ is a completely regular ordered semigroup and
let $a \in S $. Then there is $ t \in S$ such that $a \leq a^2ta^2$.
Now $$ a  \leq a^3ta^2ta^2 \leq  a^3ta^2ta^2ta^3 \leq aa'a,
\;\textrm{where} \;a'= a^2 ta^2 ta^2 ta^2.$$ Also $a' \leq a'aa'$.
Thus $a' \in  V_{\leq}(a)$.  Likewise
$$ aa'\leq a^2 ta^2a' \leq a^2ta^2 ta^3a' \leq a^2 ta^2 ta^2ta^4 a'= \;a'ua, \;\textrm{where} \;u=a^4 ta^2
ta^2 ta \in S.$$ Similarly there is $v \in S$ such that $a'a \leq
ava'$.

Conversely, suppose that each $a \in S$ satisfies the given
conditions. Consider $a \in S$. Then there  is $a' \in V_{\leq}(a)$
 and $ u, v \in S$ such that $$a \leq aa'a, \;a'a \leq aua', \;\textrm{and} \;aa'
\leq a'va.$$ This implies $a \leq aa'a \leq a a'a a' aa'a \leq a^2u
a'^3 va^2,$ and so $a \in (a^2 S a^2]$. Hence $S$ is a completely
regular ordered semigroup.
\end{proof}

\subsection{Structure of completely regular ordered semigroups}
Now we study the structure of completely regular ordered semigroups.
We show that every completely regular ordered semigroup is a union
of group like ordered semigroups.
\begin{Lemma}\label{cr12}
Let $S$ be a completely regular ordered semigroup. Then every
$\hc$-class  is an ordered  subsemigroup. Moreover   if  $H$ is  an
$\hc$-class then for every  $a \in H$  there is $h \in H$ such that
$$ a\leq aha, \;a \leq a^2h, \;\textrm{and} \;a \leq ha^2.$$
\end{Lemma}
\begin{proof}
First suppose that   $H$ is  an $\hc$-class, and $a, b \in H$. Then
$a \hc b$. Then  there are $x, y, u, v \in S$ such that $$a \leq xb,
\; b \leq ya, a \leq bu \;\textrm{and} \;b \leq va.$$ Also there are
$a' \in V_{\leq}(a)$ and $b' \in V_{\leq}(b)$ such that
$$bb' \leq b'x_1 b \;\textrm{and} \;a'a \leq ax_2a'$$ for some $x_1, x_2 \in
S$. Now $a \leq xb$ implies that $ab \leq (xb)b$ and from the
definition of ordered inverse we have  $b \leq bb'b$ which implies
$b \leq (b'x_1 b) b \leq (b' x_1y)ab$.  Thus $ab \lc b$.

In a similar manner it can be proved  that $ab \rc a$. So $ab \hc
b$. Thus $ab \in H$ and so  $H$ is an ordered subsemigroup of $S$.

The latter part is fairly straightforward.  $h = a'$ as in Theorem
\ref{cr11} serves our purpose.
\end{proof}

\begin{Theorem}\label{cr13}
In an ordered semigroup $S$,  the following conditions are
equivalent:
\begin{enumerate}
 \item \vspace{-.4cm}
$S$ is completely regular;
  \item \vspace{-.4cm}
each $\hc$-class is a group like ordered semigroup;
  \item \vspace{-.4cm}
$S$ is union of group like ordered semigroups.
\end{enumerate}
\end{Theorem}
\begin{proof}
$(1) \Rightarrow (2)$: Let $H$ be an $\hc$-class in $S$. Then $H$ is
a subsemigroup of $S$. Consider two elements $a, b \in H$. Then
there are $x, y, u, v \in S$ such that $$a \leq xb, \;b \leq ya, \;a
\leq bu \;\textrm{and} \;b \leq av.$$

Also there are $s, t \in S$ such that $$a \leq a^2 t a^2
\;\textrm{and} \;b \leq b^2 s b^2.$$ Then we have
\begin{align*}
  a &\leq a^2 t a^2\\
&\leq a^2 ta^2 ta^3\\
&\leq a^2 ta^2 ta^2 xb\\
&\leq (a^2 ta^2 ta^2x b^2 sb^2 sb^2 )b\\
&\leq hb,
  \;\textrm{where} \;h = a^2 ta^2 ta^2x b^2 sb^2 sb^2  \in S.
\end{align*}
Now $h = a(a ta^2 ta^2x b^2 sb^2 sb^2 )$ shows that $h \rc a$ and so
$h \rc b$.\\
Also
\begin{align*}
 b &\leq b^2 s b^2\\
&\leq b^3 sb^2 sb^2\\
&\leq by ab sb^2 sb^2\\
&\leq by(a^2 ta^2 ta^3 bs b^2 sb^2 )\\
&\leq by(a^2 ta^2 ta^2 xb^2 sb^2
sb^2)\\
&\leq byh.
\end{align*}
This  shows that $h \lc b$. Thus $h \in H$ and such that $a \leq
hb$. Similarly there is $h' \in H$ such that $a \leq bh'$. Hence $H$
is a group like ordered semigroup.

$(2)\Rightarrow (3)$ and $(3)\Rightarrow (1)$: These  are obvious.
\end{proof}

We now focus to  the  group like ordered subsemigroups  that contain
an ordered idempotent.

\begin{Lemma}\label{cr0.1}
Let $S$ be an  ordered semigroup and $Gr_\leq(S)\neq \phi$. Then for
every $a \in Gr_\leq (S)$ there is $e \in E_\leq(S)$ such that $a
\leq ea, \;a \leq ae$.
\end{Lemma}
\begin{proof}
Consider $a \in Gr_\leq (S)$. Then there is $t\in S$ such that
\begin{align*}
a &\leq a^2 ta^2\\
  & \leq a(a^2 ta^2 ta^2)= ae, \;\textrm{where} \;e= a^2 ta^2
ta^2 \in E_\leq(S).
\end{align*}
 Similarly $a \leq ea$.
\end{proof}

Next lemma is  straight forward that follows similarly to the above
lemma.

\begin{Lemma}\label{cr0.3}
 Let $S$ be an  ordered semigroup and $e \in E_\leq(S)$. Then
for every $a \in eSe$,  $a \leq ea \;and \;a \leq ae$.
\end{Lemma}

Let $S$ be a completely regular ordered semigroup. For  $e \in
E_\leq(S)$ let us construct the set $$G_e= \{a \in S: a \leq ea, \;a
\leq ae \;\textrm{and} \;e \leq za, \;e \leq az \;\textrm{for some}
\;z\in S\}.$$

Now  $e \in G_e$ implies $G_e$ is nonempty.
\begin{Lemma}\label{cr0.2}
Let $S$ be a completely regular ordered semigroup. Then for every $a
\in S$ there is $e \in E_\leq(S)$ and $z \in G_e$ such that  $a \leq
ea, \;a \leq ae$ and  $ \;e \leq za, \;e\leq az$.
\end{Lemma}
\begin{proof}
Let $a \in S$. Then there is $e= a^2 ta^2 ta^2 \in E_\leq(S)$ such
that $a \leq ea $ and $a \leq ae$, by Lemma \ref{cr0.1}.

Also $$e= a^2 ta^2 ta^2 \leq (a^2 ta^2 ta^2 ta^2) a$$ and likewise
$a \leq a(a^2 ta^2 ta^2 ta^2)$. Denote  $z= a^2 ta^2 ta^2 ta^2$.
Then $e \leq za$ and similarly $e \leq az$.

To prove  $z= a^2 ta^2 ta^2 ta^2 \in G_e$ we can see that
\begin{align*}
z &= a^2 ta^2 ta^2 ta^2\\
  &\leq a^2 ta^2 ta^2 ta^3ta^2\\
  &=zata^2\\
  &\leq za^2 ta^2 ta^2=ze.
\end{align*}
  Similarly $z \leq ez$. This completes the proof.
\end{proof}
As a consequence of previous lemma  we have the following theorem
that states that for every ordered idempotent in a completely
regular ordered semigroup there is a group like ordered
subsemigroup.
\begin{Theorem}\label{2.1}
Let $S$ be a completely regular ordered semigroup. Then for  every
$e \in E_\leq(S)$ the set $G_e= \{a \in S: a \leq ea, \;a \leq ae
\;and \;e \leq za, \;e \leq az \;for \;some \;z\in S\}$ is a group
like subsemigroup of $S$.
\end{Theorem}
\begin{proof}
First choose $a, b \in G_e$. Now $ab \leq a(be)= (ab)e$, similarly
$ab \leq e(ab)$. Also for $a,b \in G_e$ there are $z_1, w_1 \in S$
such that $$e \leq z_1b, \;e \leq bz_1, \;e \leq w_1 a
\;\textrm{and} \;e \leq aw_1.$$ So $e \leq z_1b$ implies that
\begin{align*}
e &\leq z_1(eb)\\
  & \leq z_1(w_1a)b\\
  &= (z_1w_1) ab\\
  &=h_1 ab, \;\textrm{where} \;h_1=
z_1 w_1.
\end{align*}
Similarly $e \leq ab h_1$.  Therefore $ab \in G_e$ and so $G_e$ is a
subsemigroup of $S$.

To show that $G_e$ is  group like,  let us choose $x, y \in G_e$.
Then from $x \leq xe$. Then there is $s \in S$ such that $e \leq sy$
and $e \leq es$. Thus  $x \leq xs y \leq x(ese)y$. Now $ese \leq
e(ese)$ and $e \leq (ese)e$. Also $e \leq e^2 \leq esy \leq (ese)y$
and $e \leq e^2 \leq yse\leq y(ese)$. This shows that $ese \in G_e$
and so $x\in (G_ey]$. Also $x\in (yG_e]$ follows dually.   Hence
$G_e$ is a group like ordered subsemigroup of $S$.
\end{proof}

Now we characterize the complete semilattice decomposition of a
completely regular ordered semigroup.
\begin{Lemma}\label{cr14}
Let $S$ be a completely regular ordered semigroup. Then $\jc$ is the
least complete semilattice congruence on  $S$.
\end{Lemma}
\begin{proof}
Let $a \in S$. Then $a \leq a^2 t a^2$ for some $t \in S$. This
implies  $a \leq a^2 ta^2 ta^3$. Thus $a \jc a^2$. Let $a, b \in S$,
then $ab \jc abab$ gives  that $$(SabS] = (SababS] \subseteq
(SbaS].$$ Interchanging the roles of $a$ and $b$ we get $(SbaS]
\subseteq (SabS]$. Thus $(SbaS] =(SabS]$ and so $ab \jc ba$.

Now let $a, b \in S$ be such that $a \jc b$ and $c \in S$. Then
there are $u, v, x, y \in S$ such that $$a \leq ubv \;\textrm{and}
\;b \leq xay.$$ Then $ac \leq ubvc$ implies that
\begin{align*}
(SacS] &\subseteq (SbvcS]\\
&=(SvcbS]\\
&\subseteq (ScbS]\\
&= (SbcS],
\end{align*}
and similarly $bc \leq xayc$ yields  that $(SbcS] \subseteq (SacS]$.
Hence $(SacS] = (SbcS]$ and so $ac \jc bc$. Thus $\jc $ is a
semilattice congruence on $S$.

Next  consider $a, b \in S $ such that $ a \leq b$, then $a^2 \leq
ab$. This   implies $(Sa^2S] \subseteq (SabS]$. So $(SaS]= (Sa^2S]
\subseteq (SabS] \subseteq (SaS]$. Thus $(SaS]= (SabS]$, that is, $a
\jc ab$.

To prove the minimality of $\jc$, as complete semilattice
congruence, consider a complete congruence $\xi$  on $S$, and
consider $a, b \in S$ such that $a \jc b$. Then $a \leq xby$ and $b
\leq uav$, for some $x, y, u, v \in S$. This  implies
\begin{align*}
a &\leq xby\\
  & \leq xuavy\\
  & \leq xuxbyvy
\end{align*}
and similarly  $b \leq uxuavyv$. Then $a \;\xi \;axuxbyvy
\;\textrm{and} \;b \;\xi \;buxuavyv $, by completeness of $\xi$.
Since $\xi$ is a complete semilattice congruence, it follows that
$$axuxbyvy \;\xi \;axyuvb \;\xi \;buxuavyv.$$
Hence $a \xi b$ and  thus $\jc$ is the least complete semilattice
congruence on $S$.
\end{proof}
\begin{Theorem}\label{cr15} (Clifford) An ordered semigroup $S$ is completely
regular if and only if it is a complete semilattice of  completely
simple ordered semigroups.
\end{Theorem}
\begin{proof}
Assume  that $S$ is a completely regular ordered semigroup. Then, by
Lemma \ref{cr14}, $\jc$ is the least complete semilattice congruence
on $S$ and so each $\jc$-class is a subsemigroup. Consider a
$\jc$-class $J$ and  $\;a, \;b \in J$. Then $a \;\jc \;b$, and so
there are $x, y \in S$ such that $a \leq xby$. Since $S$ is a
completely regular ordered semigroup, there is $u \in S$ such that
$a \leq a^2ua$, which implies that
\begin{equation}\label{eq1}
a \leq (xbyx) b (yuxby).
\end{equation}
Since $\jc$ is  complete semilattice  congruence on $S$, $a \leq
xby$ implies that
\begin{align*}
(a)_{\jc} &= (axby)_{\jc}\\
          & = (xby)_{\jc}\\
          &= (xbyx)_{\jc}
\end{align*}
and  from   (\ref{eq1}) we have
 $(a)_{\jc}= (yuxby)_{\jc}$. Thus
 $xbyx, \;yuxby \in J$ and hence  $J$ is a simple ordered semigroup.

Now to show that $J $ is completely regular ordered semigroup,
consider  $a \in \jc$. Since $\hc \subseteq \jc, \;(a)_{\hc}
\subseteq J$. Also $(a)_{\hc}$ is a group like ordered semigroup,
and so $a$ is completely regular element in $J$. Thus $J$ is a
completely simple semigroup and hence $S$ is a complete semilattice
of completely simple ordered semigroups.

The converse is obvious.
\end{proof}

\section{Clifford ordered semigroups}
Ordered  semigroups which are complete semilattices of group like
ordered semigroups are the analogue of Clifford semigroups. Though
it is not under the name Clifford ordered semigroups, but such
ordered semigroups have been studied extensively by  Kehayopulu
\cite{Ke MJ 92} and Cao \cite{CX2000} \cite{Cao2000}, specially
complete semilattice decomposition of such semigroups. Here we show
that a regular ordered semigroup $S$ is Clifford ordered semigroup
if and only if it is a complete semilattice of group like ordered
semigroups. Also a natural analogy between Clifford ordered
semigroup and Clifford semigroup has been given here. These supports
the terminology of Clifford ordered semigroup.

\begin{Definition}
Let $S$ be a regular ordered  semigroup. Then $S$ is called a
Clifford ordered semigroup \index{ordered  semigroup!Clifford}if for
all $a \in S$ and $e \in E_{\leq}(S)$ there are $u, v \in S$ such
that $ae \leq eua \;and \;ea \leq ave$.
\end{Definition}
As an immediate example of such semigroups we can consider  group
like ordered semigroups.

The following theorem states different equivalent conditions  for
the Clifford ordered semigroups.
\begin{Theorem}\label{cr26}
Let $S$ be a regular ordered  semigroup. Then the following
conditions are equivalent:
\begin{enumerate}
\item \vspace{-.4cm}
$S$ is Clifford;
\item \vspace{-.4cm}
$\lc = \rc$;
\item \vspace{-.4cm}
$(aS] = (Sa]$ for all $a \in S$;
\item \vspace{-.4cm}
$(eS] = (Se]$ for all $e \in E_\leq(S)$.
\end{enumerate}
\end{Theorem}
\begin{proof}
$(1)\Rightarrow(2):$ Consider $a, b \in S$ such that $a \lc b$. Then
there are $x, y \in S$ such that $a \leq xb \;\textrm{and} \;b\leq
ya$. Since $S$ is regular $a \leq aza \;\textrm{and} \; b \leq bwb$
for some $z,w \in S$. Then $a \leq xbwb$. Since  $bw \in E_\leq(S)$
and  $S$ is  Clifford,  we obtain that $xbw \in (bw S x]$, and hence
$a \in (bS]$. Similarly  $b \in (aS]$, which implies  $a \rc b$.
Therefore $\lc \subseteq \rc$.

$\rc \subseteq \lc$ follows dually. Hence $\lc= \rc$.

$(2)\Rightarrow (3):$  Let $a \in S$. Choose $b \in (aS]$.  Then
there is $s \in S$ such that $b\leq as$. The regularity  of $S$
yields that $as \leq ast ast as= x_1 s$ where $x_1= ast ast a$. Then
$as \rc x_1$ implies  $as \lc x_1$, by condition (2). Then there is
$z \in S$ such that $as \leq zx_1$. Therefore $b \leq as \leq zx_1=
zast ast a$, so $b \in (Sa]$. Hence  $(aS] \subseteq (Sa]$.

$(Sa] \subseteq (aS]$ follows dually. Hence $(aS] =(Sa]$.

$(3)\Leftrightarrow (4):$    First suppose that condition (4) holds
in $S$. Consider $a\in S$. Let $z\in (aS]$ then there is $t \in S$
such that $z\leq at$. Since $S$ is regular there is $x \in S$ such
that $a \leq axa$. Clearly $xa \in E_\leq(S)$. Then  $z \leq at$
implies that $a \leq axat$. Since $xa \in E_\leq(S)$ we have $xat\in
(Sxa]$, by condition (4). Therefore $z \in (Sa]$ and so $(aS]
\subseteq (Sa]$.

$(Sa] \subseteq (aS]$ follows dually. Hence $(aS] = (Sa]$.

Converse is obvious.

$(4)\Rightarrow (1):$  Let $a \in S$ and $e \in E_\leq(S)$. Since
$S$ is regular there is $x \in V_\leq(ae)$ such that $ae\leq ae
xae\leq aee xae$. Now there  are $t \in S$ such that $ae\leq et$, by
 condition (4). Also $exa \leq exaexa$, that is, $exa \in E_\leq(S)$.
Then for $exa \in E_\leq(S)$ and $e \in S$ there is $s\in S$ such
that $exae \leq sexa$, using condition (4).  Hence $ae \leq etsxa= e
za$, where $z= tsx$. Similarly $ea \leq awe$ for some $w \in S$.
Hence $S$ is Clifford.
\end{proof}

The underlying spirit of these results can be realized from the
following theorem.
\begin{Theorem}\label{cr27}
Let $S$ be a regular ordered  semigroup. Then $S$ is Clifford if and
only if for all $ a, b \in S$ there is $x \in S$  such that $ab \leq
bxa$, i.e. $S$ is $\hc$-commutative.
\end{Theorem}
\begin{proof}
Let   $ S$ be Clifford and  $a, b \in S $. Since $S$ is regular, so
there are $x, y, z \in S$ such that
$$a \leq axa, \;b \leq byb \;\textrm{and} \;ab \leq abzab.$$
This implies $by \;\textrm{and}  \;xa$ are ordered  idempotents and
hence there are $u, v \in S$ such that $$aby \leq byua \;\textrm{
and} \;xab \leq bvxa.$$ Now  $ab \leq abzab \leq aby bza xab$
implies $ab \leq b(yua bza bvx)a$.

Converse follows directly.
\end{proof}
Another application of this theorem is that every Clifford ordered
semigroup is completely regular. For, consider a Clifford ordered
semigroup $S$ and $a \in S$. Then there is  $x \in S$ such that $a
\leq axa xa xa$. Also there are $u, v \in S$ such that $$xa \leq aux
\;\textrm{and} \;ax \leq xva,$$ which again implies that $ a \leq
a^2ux^3 va^2$, and hence   $a \in  (a^2 S a^2]$. Thus $S$ is a
completely regular ordered semigroup.  But the converse is not true
in general. The condition for which a completely regular ordered
semigroup becomes a Clifford ordered semigroup has been given in the
following theorem.
\begin{Theorem}\label{cr29}
An  ordered semigroup $S$ is a  Clifford ordered semigroup if and
only if  $S$ is completely regular ordered semigroup and  for all
$e, f \in E_\leq (S), \;eSf \subseteq (f S e]$.
\end{Theorem}
\begin{proof}
First suppose  that the  given conditions hold in  $S$. Let $a \in S
\;\textrm{and} \;e \in E_\leq(S)$. Since $S$ is completely regular
there is $x \in S$ such that $a \leq a^2xa \;\textrm{and} \;a \leq
axa^2$, by Lemma \ref{cr10.1}. Then $xa^2, a^2x \in E_\leq (S)$. Now
$ae \leq axa^2e$. Since $xa^2, e \in E_\leq (S)$ there is $s_1 \in
S$ such that $xa^2e \leq es_1xa^2$, by given condition. Therefore
$ae \leq aes_1 xa^2$. Also $aes_1xa^2 \leq a^2xa es_1xa^2$. Since
$a^2x, e \in E_\leq (S)$, by given condition it follows that $a^2xae
\leq es_2a^2x$ for some $s_2 \in S$. Thus $ae \leq aes_1 xa^2$
implies $ae \leq es_2a^2x s_1xa^2$, that is $ae \in (eSa]$.
Similarly $ea \in (aSe]$. Hence $S$ is a Clifford ordered semigroup.

Converse is obvious.
\end{proof}

\begin{Theorem}
An ordered semigroup $S$ is Clifford if and only if it is
completely regular and inverse.
\end{Theorem}
\begin{proof}
First suppose that $S$ is both completely regular and inverse. Consider $e, f
\in S$. Consider $exf$ for some $x \in S$. Since $S$ is completely
regular there is  $z \in V_\leq(exf)$ such that  $exf \leq exf z
exf, \;exf \leq zexfexf$ and $exf \leq exfexf z$, by Lemma
\ref{cr12}. Now $fze \leq fzexfze \leq fze(exf)fze$ and $exf \leq
exf z exf \leq exf(fze)exf$. This shows that $fze \in V_\leq(exf)$.
Also $exf \leq exf z exf \leq exf (exf z^2) exf$ and $exfz^2 \leq
exf z exf z^2 \leq exfz zexf exf z^2=exfz^2 (exf) exfz^2$, which
implies that $exfz^2 \in V_\leq(exf)$. Similarly $z^2exf\in
V_\leq(exf)$.

Since  $exfz^2, fze \in V_\leq(exf)$ by given condition we have
$exfz^2\leq fzet $ for some $t \in S$. Thus $exf \leq exf z exf \leq
exf z^2 (exf)^2$ implies that $exf \leq fzet(exf)^2= fsexf$, where
$s=zetexf$. Similarly for $z^2exf, fze\in V_\leq(exf)$ we have $exf
\leq exf ufze$. Therefore $exf \leq fve$ for some $v \in S$. So by
Theorem \ref{cr29} we have $S$ is Clifford.

Conversely, assume that $S$ is Clifford ordered semigroup. Clearly
$S$ is a completely regular ordered semigroup. Let $e,f \in
E_\leq(S)$. Then $ef \leq eef$. Since $S$ is Clifford we have $eef
\in (fSe]$, by Theorem \ref{cr29}. So $ef \in (fSe]$. Hence the
given condition follows from Theorem \ref{cr6}.
\end{proof}

\begin{Theorem}\label{cr28}
Let $S$ be an ordered semigroup. Then $S$ is a Clifford ordered
semigroup if and only if it is a complete semilattice of  group like
ordered semigroups.
\end{Theorem}
\begin{proof}
Let $S$ be a Clifford ordered semigroup. Then $S$ is completely
regular and hence $\jc$ is the least complete semilattice congruence
on $S$ and each  $\hc$-class is a group like ordered semigroup, by
Theorem \ref{cr14} and Theorem \ref{cr13}. Now $\lc = \rc$ implies
that $\jc = \hc$ and so $S$ is a complete semilattice of group like
ordered semigroups.

Conversely, suppose that  $S$ is  the complete semilattice  $Y$ of
group like ordered semigroups $\{G_{\alpha}\}_{\alpha \in Y}$.
Consider $a, b \in S$. Then $ab, ba \in G_{\alpha}$ for some $\alpha
\in Y$, and so there are $x, y, z \in G_{\alpha}$ such that $$ab
\leq xba, \;ab \leq bay, \;\textrm{and} \;ab \leq abzab,$$ which
together implies that $ab \leq bay zxba$. Thus $S$ is Clifford
ordered semigroup.
\end{proof}

\subsection{Left Clifford ordered semigroups} In this section we
introduce left Clifford ordered semigroups which are of course a
generalization of Clifford ordered semigroups. Here we show that a
left Clifford ordered semigroup is  a complete semilattice of left
group like ordered semigroups.
\begin{Definition}
A regular ordered  semigroup $S$ is called a left Clifford
\index{ordered  semigroup!left Clifford} ordered semigroup if for
all $a \in S$,  $(aS] \subseteq (Sa]$.
\end{Definition}
Every left group like ordered semigroup is a left Clifford ordered
semigroup.
\begin{Theorem}\label{cr30}
Let $S$ be a regular ordered  semigroup. Then the following
conditions are equivalent:
\begin{enumerate}
\item \vspace{-.4cm}
$S$ is a left Clifford ordered semigroup;
\item \vspace{-.4cm}
for all $e \in E_{\leq}(S), \;(eS] \subseteq (Se]$;
\item \vspace{-.4cm}
for all $a \in S, \;and \;e \in E_{\leq}(S)$ there is $x\in S$ such
that $ea \leq xe$;
\item \vspace{-.4cm}
for all $a, b \in S$ there is $x\in S$ such that $ab \leq xa$;
\item \vspace{-.4cm}
$\rc \subseteq \lc$ on $S$.
\end{enumerate}
\end{Theorem}
\begin{proof}
$(1)\Rightarrow (2)$ and $(2)\Rightarrow (3)$ are trivial.

$(3)\Rightarrow (4):$ Let $a, b \in S$. Since $S$ is regular, so
there is $x \in S$ such that $a \leq axa$. Then  $xa \in
E_{\leq}(S)$, and whence $xab \leq uxa$ for some $u \in S$. So $$ab
\leq axab \leq (aux)a.$$

$(4)\Rightarrow (5):$ Let $a, b \in S$ be such that $a \rc b$. Then
 there are $s, t \in S$ such that $$a \leq bs \;\textrm{and} \;b \leq at.$$
 Also $$bs \leq xb \;\textrm{and }\;at \leq ya  \;\textrm{for some} \;x, y \in S.$$
This implies  $$a \leq xb \;\textrm{and} \;b \leq ya.$$ Thus $a \lc
b$, and hence $\rc \subseteq \lc$ on $S$.

$(5)\Rightarrow (1):$ Let $a \in S$ and $u \in (aS]$. Then there is
$s \in S$ such that $u \leq as$. Since $S$ is regular,  $as \leq as
t as$ for some $t \in S$. Then $as \rc asta$ and so $as \lc asta$.
This implies $as \leq xasta $ for some $x\in S$, so $u \leq xasta$.
Then $u \in (Sa]$ and hence $(aS] \subseteq  (Sa]$.
\end{proof}

\begin{Theorem}\label{cr34}
An ordered semigroup $S$ is a left Clifford ordered semigroup if and
only if the following conditions hold in $S$:
\begin{itemize}
\item  [(i)] $a \in (a S a^2]$  for every $a \in S$,
\item[(ii)]
$ef \in (ef S fe]$ for every $e, f \in E_\leq (S)$.
\end{itemize}
\end{Theorem}

\begin{proof}
First suppose that $S$ is a left Clifford ordered semigroup. Let $a
\in S$. Then there is $x \in S$ such that   $a \leq axa$, by the
regularity of $S$. This implies $a \leq axa xa $. Since $S$ is left
Clifford ordered semigroup,  there is $x_1 \in S$ such that $ax \leq
x_1a $. So from $a \leq ax a xa$ we have $a \leq ax x_1a^2$. Thus $a
\in (aSa^2]$.

Next consider $e, f \in E_\leq (S)$. Then there is $x_2  \in S$ such
that
$$ef \leq ef x_2 ef ef, \;\textrm{by condition (i)}.$$
Now  there is $x_3 \in S$ such that
$$fef \leq x_3 fe, \;\textrm{ since} \;S \;\textrm{is left Clifford}.$$
Then $ef \leq ef x_2 e x_3 fe$ and so $ef \in (ef S fe]$.

Conversely assume that the given conditions hold in $S$. Let $a \in
S$ and $e \in E_\leq(S)$. Then by condition (i) there is $x \in S$
such that $a \leq axa^2$. This implies $a^2 \leq a^2 xa^2$ and  so
$xa^2 \in E_\leq (S)$. Then  there is $u \in S$ such that $xa^2 e
\leq xa^2 e u e xa^2$. Now $ae  \leq axa^2e$ implies that $ae \leq
axa^2 eue xa^2$. Thus $ae \in (Sa]$.
\end{proof}
Now we characterize the complete semilattice decomposition of left
Clifford ordered semigroups.
\begin{Theorem}\label{cr32}
Let $S$ be an ordered semigroup. Then $S$ is left Clifford ordered
semigroup if and only if $\lc$ is the least complete semilattice
congruence on $S$.
\end{Theorem}
\begin{proof}
Let $S$ be a  left Clifford ordered semigroup. Consider $a, b \in S$
such that $a \lc b$ and $c \in S$. Then  there are $s_1, s_2 \in S$
such that $$a \leq s_1b \;\textrm{and} \;b \leq s_2a.$$ Also $$cs_1
\leq uc \;\textrm{and} \;cs_2 \leq vc \;\textrm{for some} \;u, v \in
S.$$

This implies $ca \leq cs_1b \leq ucb$, thus  $ca \leq ucb $ and
similarly $cb \leq vca$. Therefore $ca \lc cb$. Hence $\lc$ is a
congruence on $S$.

Now let $a, b \in S$. Then there is $x \in S$ such that $ab \leq ab
x ab$. Also
$$ab \leq u_1a \;\textrm{and} \;bxu_1 \leq u_2b \;\textrm{for some} \;u_1, u_2 \in S.$$ Then we
have
\begin{align*}
ab &\leq ab x ab\\
   & \leq ab x u_1a\\
   & \leq au_2ba.
\end{align*}
This implies that $(Sab] \subseteq (Sba]$. Interchanging the role of
$a, b $ we get $(Sba] \subseteq (Sab]$. Thus $(Sba] = (Sab]$ and $ab
\lc ba$. Again $a \leq aya $ for some $y \in S$. This implies that
\begin{align*}
(Sa] &\subseteq (Saya]\\
     & \subseteq (Sya^2]\\
     & \subseteq (Sa^2].
\end{align*}
Thus $a\lc a^2$.

Next  consider $a, b \in S$ be such that $a \leq b$. Now there is $z
\in S$ such that  $a \leq aza $ and so $ab \leq (az)ab$. Since $S$
is a left Clifford ordered semigroup, $ab \leq va$ for some $v \in
S$, whence $a \lc ab$. Thus $\lc$ is the complete semilattice
congruence on $S$.

Let $\rho$ be a complete semilattice congruence on $S$ and $a,b \in
S$ be such that $a \lc b$. Then there are $t_1, t_2 \in S$ such that
$$a \leq t_1b \;\textrm{and} \;b \leq t_2a.$$
Then $a \leq xyxb \;\textrm{and} \;b \leq yxya$,  and then by the
completeness of $\rho$, it follows that $(a)_\rho= (at_1b)_\rho
\;and \; (b)_\rho= (bt_2a)_\rho$. This implies
\begin{align*}
(a)_\rho &= (at_1b)_\rho\\
         &= (at_1)_\rho (b)_\rho\\
         &= (at_1)_\rho (bt_2a)_\rho\\
         &= (at_1bt_2a)_\rho\\
         &= (at_1b)_\rho (bt_2a)_\rho\\
         &= (a)_\rho (bt_2a)_\rho\\
         &= (bt_2a)_\rho\\
         &= (b)_\rho.
\end{align*}  Thus $a \rho b$ and hence $\lc$ is the least complete semilattice congruence on
$S$.

The converse of this theorem follows trivially.
\end{proof}
\begin{Theorem}\label{cr33}
Let $S$ be a regular ordered  semigroup. Then $S$ is  a left
Clifford ordered semigroup if and only if it is a complete
semilattice of left group like ordered semigroups.
\end{Theorem}
\begin{proof}
Let $S$ be a  left Clifford ordered semigroup. In view of Theorem
\ref{cr32} it is sufficient to show that  each  $\lc$-class is a
left group like ordered semigroup. Let $L$ be an $\lc$-class in $S$.
Then $L$ is a subsemigroup of $S$, since $\lc $ is a  complete
semilattice congruence on $S$. Let $a, b \in L$. Then there are $s,
t, x \in S$ such that $$a \leq xb, \; a \leq asa \;\textrm{and} \;b
\leq btb.$$ This implies  $a \leq asxb \leq (asxbt)b =ub,
\;\textrm{where} \;u = asxbt$.

Since  $\lc$ is complete semilattice congruence on $S$, we have
$$a \lc \;aub \lc \;a^2s xb tb \lc \;as xbt \;=u. $$ This shows that
$u \in L$.  Thus $L$ is left group like ordered semigroup.

Conversely, let $\rho$ be a  complete semilattice  congruence on $S$
and each $\rho$-class is a left group like ordered semigroup.
Consider $a, b \in S$.  Since $\rho$ is a complete semilattice
congruence on $S$, $ab \rho ba$ and hence $ab, ba$ are in the left
group like ordered semigroup $(ab)_\rho$. So $ab \leq xba$ for some
$x \in S$.  Hence  $S$ is a left Clifford ordered semigroup.
\end{proof}

Characterization of right Clifford  ordered semigroups can be done
dually.

\bibliographystyle{plain}

\end{document}